\documentclass{article}

\usepackage{amsthm,amsmath}
\usepackage{graphicx}
\usepackage{enumerate}
\usepackage{amssymb} 
\usepackage{amsfonts,epsfig,epstopdf,url,array}

\RequirePackage[numbers]{natbib}
\RequirePackage[colorlinks,citecolor=blue,urlcolor=blue]{hyperref}

\theoremstyle{remark}
\newtheorem{remark}{Remark}

\newtheorem{theorem}{Theorem}

\newtheorem{example}{Example}

\def\Item$#1${\item {\hfill $\displaystyle#1$ \hfill\refstepcounter{equation}(\theequation)}}

\begin{document}

\title{No Perfect Cuboid}
\author{Walter Wyss}
\date{}
\maketitle

\begin{abstract}
A rectangular parallelepiped is called a cuboid (standing box). It is called perfect if its edges, face diagonals and body diagonal all have integer length. Euler gave an example where only the body diagonal failed to be an integer (Euler brick). Are there perfect cuboids? We prove that there is no perfect cuboid.
\end{abstract}

\section{Introduction}

Cuboids have been studied extensively. It suffices to look at rational cuboids \cite{leech77,leech81}. Rational cuboids are characterized by seven positive rational numbers (three different edges, three different face diagonals and the body diagonal). Examples are known where all but one of the seven quantities are rational. 
Our approach uses the concept of a rational leaning box. This is a parallelepiped with two different rectangular faces and a face that is a parallelogram. Rational leaning boxes are characterized by nine positive rational numbers (three different edges, two different face diagonals belonging to the rectangular faces, two different face diagonals belonging to the face parallelogram and two different body diagonals).
If the parallelogram face becomes a rectangle, then we have a standing box. Computer aided discoveries have shown the existence of perfect leaning boxes \cite{sawyer, benjamin, rathbun}.
We found a two-parameter family of solutions for rational leaning boxes analytically. The two diagonals of the face parallelogram can never be equal. Thus there is no standing rational box in this family. Finally we use an equivalent description of leaning boxes to show that in general there is no perfect cuboid. \\
In the appendices we use generic symbols which do not necessarily coincide with the ones used in the main text.

\section{The equations for the perfect leaning box}

All the following nine quantities are positive integers. \\
x,y,z, denote the three different edges.
The face rectangle (x,y) has diagonal a, the face rectangle (x,z) 
has diagonal b and the face parallelogram (y,z) has diagonals $c_1, c_2$. 
The two different body diagonals are denoted by $d_1, d_2$. \\

These quantities satisfy the equations \\

\begin{align}
x^2+y^2&=a^2 \\
x^2+z^2&=b^2 \\
x^2+c_1^2&=d_1^2 \\
x^2+c_2^2&=d_2^2 \\
2y^2+2z^2&=c_1^2+c_2^2
\end{align}

The last equation represents a perfect parallelogram \cite{wyss12, wyss14} \\

\section{Parameterization for the rational leaning box}

We look for solutions of the equations (1), (2), (3), (4), (5) in rational positive numbers. \\

We now scale these equations as follows

\begin{align}
&u_1=\frac{y}{x},& &u_2=\frac{z}{x},& &u_3=\frac{c_1}{x},& &u_4=\frac{c_2}{x}& \\
&v_1=\frac{a}{x},& &v_2=\frac{b}{x},& &v_3=\frac{d_1}{x},& &v_4=\frac{d_2}{x}&
\end{align} \\

Then $u_k$ and $v_k$, k=1,2,3,4, are positive rational numbers. \\

The scaled equations are

\begin{align}
1+u_1^2 &= v_1^2 \\
1+u_2^2 &= v_2^2 \\
1+u_3^2 &= v_3^2 \\
1+u_4^2 &= v_4^2 \\
2u_1^2+2u_2^2 &= u_3^2+u_4^2
\end{align} \\

The last equation represents a rational parallelogram (Appendix C) \\

The scaled equations can be parameterized by the four Heron angles ${\psi}_k$ and their generators $s_k$,
k=1,2,3,4 (Appendix A) as follows

\begin{align}
u_k=\cot{\psi}_k=\frac{1-s_k^2}{2s_k}=\frac{1}{2}(\frac{1}{s_k}-s_k) \\
v_k=\frac{1}{\sin{\psi_k}}=\frac{1+s_k^2}{2s_k}=\frac{1}{2}(\frac{1}{s_k}+s_k) \\
v_k-u_k=s_k, v_k+u_k=\frac{1}{s_k}, 0<s_k<1
\end{align}

\section{The three parallelograms}

Besides the face parallelogram there are two interior parallelograms \\

\begin{align}
&\text{I. } &&2u_1^2+2u_2^2 = u_3^2+u_4^2&& \\
&\text{II. } &&2u_1^2+2v_2^2 = v_3^2+v_4^2&& \\
&\text{III. } &&2v_1^2+2u_2^2 = v_3^2+v_4^2&& 
\end{align}

From Appendix (D.32) we have the representation

\begin{align}
&\text{I. } &2u_1 &= u_3{\omega}_+({\alpha}) + u_4{\omega}_-({\alpha}) \\
&&2u_2 &= -u_3{\omega}_-({\alpha}) + u_4{\omega}_+({\alpha}) \\
&\text{II. } &2u_1 &= v_3{\omega}_+({\alpha}_1) + v_4{\omega}_-({\alpha}_1) \\
&&2v_2 &= -v_3{\omega}_-({\alpha}_1) + v_4{\omega}_+({\alpha}_1) \\
&\text{III. } &2v_1 &= v_3{\omega}_+({\alpha}_2) + v_4{\omega}_-({\alpha}_2) \\
&&2u_2 &= -v_3{\omega}_-({\alpha}_2) + v_4{\omega}_+({\alpha}_2)
\end{align}

with m as the generator of ${\alpha}$ \\
\begin{align}
m = \frac{2u_2+u_3-u_4}{2u_1+u_3+u_4}
\end{align}

and with $m_1$ as the generator of ${\alpha}_1$ \\
\begin{align}
m_1 = \frac{2v_2+v_3-v_4}{2u_1+v_3+v_4}
\end{align}

and with $m_2$ as the generator of ${\alpha}_2$ \\
\begin{align}
m_2 = \frac{2u_2+v_3-v_4}{2v_1+v_3+v_4}
\end{align}

${\alpha},{\alpha}_1,{\alpha}_2$ are Heron angles in the first quadrant. \\

In terms of generators $s_3, s_4$, with \\

\begin{align}
Q=s_3s_4 
\end{align} \\

we get from Appendix E the representation

\begin{align}
&\text{I. } &4Qu_1 &= s_4M({\alpha})+s_3H({\alpha}) \\
&&4Qu_2 &= -s_4K({\alpha})+s_3N({\alpha}) \\
&\text{II. } &4Qu_1 &= s_4N({\alpha}_1)+s_3K({\alpha}_1) \\
&&4Qv_2 &= -s_4H({\alpha}_1)+s_3M({\alpha}_1) \\
&\text{III. } &4Qv_1 &= s_4N({\alpha}_2)+s_3K({\alpha}_2) \\
&&4Qu_2 &= -s_4H({\alpha}_2)+s_3M({\alpha}_2)
\end{align} \\

Comparing Equations (29-32) and using (15) we find \\

\begin{align}
0 &= s_4[M({\alpha})-N({\alpha}_1)] + s_3[H({\alpha})-K({\alpha}_1)] \\
8Qu_1 &= s_4[M({\alpha})+N({\alpha}_1)] + s_3[H({\alpha})+K({\alpha}_1)] \\
4Qs_2 &= s_4[K({\alpha})-H({\alpha}_1)] - s_3[N({\alpha})-M({\alpha}_1)] \\
4Q\frac{1}{s_2} &= -s_4[K({\alpha})+H({\alpha}_1)] + s_3[N({\alpha})+M({\alpha}_1)]
\end{align} \\

With

\begin{align*}
{\alpha}+{\alpha}_1 &= 2{\sigma}_1,& {\alpha}-{\alpha}_1 &= 2{\delta}_1 \\
{\omega}_+({\sigma}_1) &= \sqrt{2}\cos{\psi},& {\omega}_-({\sigma}_1) &= \sqrt{2}\sin{\psi} \\
\sqrt{2}\cos{\delta}_1 &= {\omega}_+({\alpha}+{\psi}),& \sqrt{2}\sin{\delta}_1 &= -{\omega}_-({\alpha}+{\psi})
\end{align*} \\

and Appendix E (Lemma 7) Equations (35-38) become 

\begin{align}
0 &= s_3\cos{\psi}H({\alpha}+{\psi})-s_4\sin{\psi}K({\alpha}+{\psi}) \\
4Qu_1 &= s_4\cos{\psi}M({\alpha}+{\psi})+s_3\sin{\psi}N({\alpha}+{\psi}) \\
2Qs_2 &= s_4\cos{\psi}K({\alpha}+{\psi})+s_3\sin{\psi}H({\alpha}+{\psi}) \\
2Q\frac{1}{s_2} &= -s_4\sin{\psi}M({\alpha}+{\psi})+s_3\cos{\psi}N({\alpha}+{\psi})
\end{align}

or in matrix form \\

\begin{align}
\left(
    \begin{array}{cc}
      s_4\sin{\psi} & -s_3\cos{\psi} \\
      s_4\cos{\psi} & s_3\sin{\psi}
    \end{array}
\right)
\left(
    \begin{array}{c}
      K({\alpha}+{\psi}) \\
      H({\alpha}+{\psi})
    \end{array}
\right) = 2Q
\left(
    \begin{array}{c}
      0 \\
      s_2
    \end{array}
\right) 
\end{align}

\begin{align}
\left(
    \begin{array}{cc}
      s_4\cos{\psi} & s_3\sin{\psi} \\
      -s_4\sin{\psi} & s_3\cos{\psi}
    \end{array}
\right)
\left(
    \begin{array}{c}
      M({\alpha}+{\psi}) \\
      N({\alpha}+{\psi})
    \end{array}
\right) = 2Q
\left(
    \begin{array}{c}
      2u_1 \\
      \frac{1}{s_2}
    \end{array}
\right) 
\end{align} \\

The inverse equations are

\begin{align}
\left(
    \begin{array}{c}
      K({\alpha}+{\psi}) \\
      H({\alpha}+{\psi})
    \end{array}
\right) =
\left(
    \begin{array}{cc}
      s_3\sin{\psi} & s_3\cos{\psi} \\
      -s_4\cos{\psi} & s_4\sin{\psi}
    \end{array}
\right)
\left(
    \begin{array}{c}
      0 \\
      2s_2
    \end{array}
\right) 
\end{align}

\begin{align}
\left(
    \begin{array}{c}
      M({\alpha}+{\psi}) \\
      N({\alpha}+{\psi})
    \end{array}
\right) =
\left(
    \begin{array}{cc}
      s_3\cos{\psi} & -s_3\sin{\psi} \\
      s_4\sin{\psi} & s_4\cos{\psi}
    \end{array}
\right)
\left(
    \begin{array}{c}
      4u_1 \\
      \frac{2}{s_2}
    \end{array}
\right) 
\end{align} \\

Explicitly \\

\begin{align}
K({\alpha}+{\psi}) &= 2s_2s_3\cos{\psi} \\
H({\alpha}+{\psi}) &= 2s_2s_4\sin{\psi} \\
M({\alpha}+{\psi}) &= 4u_1s_3\cos{\psi} - 2\frac{s_3}{s_2}\sin{\psi} \\
N({\alpha}+{\psi}) &= 4u_1s_4\sin{\psi} + 2\frac{s_4}{s_2}\cos{\psi}
\end{align} \\

Comparing equations 29,30,33,34 and using (15) we find \\

\begin{align}
0&=s_3[N({\alpha})-M({\alpha_2})] - s_4[K({\alpha})-H({\alpha_2})] \\
8Qu_2&=s_3[N({\alpha})+M({\alpha_2})] - s_4[K({\alpha})+H({\alpha_2})] \\
4Qs_1&=s_4[N({\alpha_2})-M({\alpha})] + s_3[K({\alpha_2})-H({\alpha})] \\
4Q\frac{1}{s_1}&=s_4[N({\alpha_2})+M({\alpha})] + s_3[K({\alpha_2})+H({\alpha})]
\end{align} \\

With \\

\begin{align*}
{\alpha}+{\alpha}_2 &= 2{\sigma}_2,& {\alpha}-{\alpha}_2 &= 2{\delta}_2 \\
{\omega}_+({\sigma}_2) &= \sqrt{2}\cos{\phi},& {\omega}_-({\sigma}_2) &= \sqrt{2}\sin{\phi} \\
\sqrt{2}\cos{\delta}_2 &= {\omega}_+({\alpha}+{\phi}),& \sqrt{2}\sin{\delta}_2 &= -{\omega}_-({\alpha}+{\phi})
\end{align*} \\

and Appendix E (Lemma 7), equations (51-54) become \\

\begin{align}
0&=-s_4\cos{\phi}K({\alpha}+{\phi})-s_3\sin{\phi}H({\alpha}+{\phi}) \\
4Qu_2&=-s_4\sin{\phi}M({\alpha}+{\phi})+s_3\cos{\phi}N({\alpha}+{\phi}) \\
2Qs_1&=s_4\sin{\phi}K({\alpha}+{\phi})-s_3\cos{\phi}H({\alpha}+{\phi}) \\
2Q\frac{1}{s_1}&=s_4\cos{\phi}M({\alpha}+{\phi})+s_3\sin{\phi}N({\alpha}+{\phi})
\end{align} \\

or in matrix form

\begin{align}
\left(
    \begin{array}{cc}
      -s_4\cos{\phi} & -s_3\sin{\phi} \\
      s_4\sin{\phi} & -s_3\cos{\phi}
    \end{array}
\right)
\left(
    \begin{array}{c}
      K({\alpha}+{\phi}) \\
      H({\alpha}+{\phi})
    \end{array}
\right) = 
\left(
    \begin{array}{c}
      0 \\
      2Qs_1
    \end{array}
\right) 
\end{align}

\begin{align}
\left(
    \begin{array}{cc}
      -s_4\sin{\phi} & s_3\cos{\phi} \\
      s_4\cos{\phi} & s_3\sin{\phi}
    \end{array}
\right)
\left(
    \begin{array}{c}
      M({\alpha}+{\phi}) \\
      N({\alpha}+{\phi})
    \end{array}
\right) = 
\left(
    \begin{array}{c}
      4Qu_2 \\
      2Q\frac{1}{s_1}
    \end{array}
\right) 
\end{align} \\

The inverse equations are \\

\begin{align}
\left(
    \begin{array}{c}
      K({\alpha}+{\phi}) \\
      H({\alpha}+{\phi})
    \end{array}
\right) =
\left(
    \begin{array}{cc}
      -s_3\cos{\phi} & s_3\sin{\phi} \\
      -s_4\sin{\phi} & -s_4\cos{\phi}
    \end{array}
\right)
\left(
    \begin{array}{c}
      0 \\
      2s_1
    \end{array}
\right) 
\end{align}

\begin{align}
\left(
    \begin{array}{c}
      M({\alpha}+{\phi}) \\
      N({\alpha}+{\phi})
    \end{array}
\right) =
\left(
    \begin{array}{cc}
      -s_3\sin{\phi} & s_3\cos{\phi} \\
      s_4\cos{\phi} & s_4\sin{\phi}
    \end{array}
\right)
\left(
    \begin{array}{c}
      4u_2 \\
      2\frac{1}{s_1}
    \end{array}
\right) 
\end{align} \\

Explicitly \\

\begin{align}
K({\alpha}+{\phi}) &= 2s_1s_3\sin{\phi} \\
H({\alpha}+{\phi}) &= -2s_1s_4\cos{\phi} \\
M({\alpha}+{\phi}) &= -4u_2s_3\sin{\phi} + 2\frac{s_3}{s_1}\cos{\phi} \\
N({\alpha}+{\phi}) &= 4u_2s_4\cos{\phi} + 2\frac{s_4}{s_1}\sin{\phi}
\end{align} \\

Finally, comparing equations (31-34), we get from equations (37,38) \\

\begin{align}
4Qu_2 &= -s_4K({\alpha})+s_3N({\alpha}) \\
4Qv_2 &= -s_4H({\alpha_1})+s_3M({\alpha_1})
\end{align} \\

and from equations (53,54) \\

\begin{align}
4Qu_1 &= s_4M({\alpha})+s_3H({\alpha}) \\
4Qv_1 &= s_4N({\alpha_2})+s_3K({\alpha_2})
\end{align} \\

Thus we are left with the equations \\

\begin{align*}
4Qu_2 &= -s_4K({\alpha})+s_3N({\alpha}) \\
4Qu_1 &= s_4M({\alpha})+s_3H({\alpha}) \\
4Qu_2 &= -s_4H({\alpha_2})+s_3M({\alpha_2}) \\
4Qu_1 &= s_4N({\alpha_1})+s_3K({\alpha_1})
\end{align*} \\

or

\begin{align*}
-s_4K({\alpha})+s_3N({\alpha}) &= -s_4H({\alpha_2})+s_3M({\alpha_2}) \\
s_4M({\alpha})+s_3H({\alpha}) &= s_4N({\alpha_1})+s_3K({\alpha_1})
\end{align*} \\

But these are precisely equations (51,35). Thus there are only the equations (47-50) and (63-66). \\

\begin{remark}
Equation 50 follows from the equations (47-49) and the identity (E.22) \\
\end{remark}

\begin{proof}
\begin{align*}
N({\alpha}+{\psi}) &= \frac{4Q+H({\alpha}+{\psi})M({\alpha}+{\psi})}{K({\alpha}+{\psi})} \\
&= \frac{4Q+4Q[2u_1s_2\sin{{\psi}}\cos{{\psi}}-\sin^2{{\psi}}]}{2s_2s_3\cos{{\psi}}} \\
&= \frac{4Q\cos{{\psi}}[2u_1s_2\sin{{\psi}}+\cos{{\psi}}]}{2s_2s_3\cos{{\psi}}} \\
&= 2u_12s_4\sin{{\psi}}+2\frac{s_4}{s_2}\cos{{\psi}}
\end{align*} \\
\end{proof} 

\begin{remark}
Equation 66 follows from equations (63-65) and the identity (E.22) \\
\end{remark} 

\begin{proof}
\begin{align*}
N({\alpha}+{\phi}) &= \frac{4Q+H({\alpha}+{\phi})M({\alpha}+{\phi})}{K({\alpha}+{\phi})} \\
&= \frac{4Q-2s_1s_4\cos{{\phi}}[-4u_2s_3\sin{{\phi}}+2\frac{s_3}{s_1}\cos{{\phi}}]}{2s_1s_3\sin{{\phi}}} \\
&= 4u_2s_4\cos{{\phi}}+\frac{2s_4}{s_1}\sin{{\phi}}
\end{align*} \\
\end{proof} 

\begin{remark}
Equations (47,48,63,64) result in the relation \\
\begin{align}
s_1s_2=\tan{({\phi}-{\psi})}
\end{align} \\
\end{remark}

\begin{proof}
\begin{align*}
&K({\alpha}+{\psi})H({\alpha}+{\phi})-K({\alpha}+{\phi})H({\alpha}+{\psi}) \\
=&-2s_2s_3\cos{{\psi}}2s_1s_4\cos{{\phi}}-2s_1s_3\sin{{\phi}}2s_2s_4\sin{{\psi}} \\
=&-4Qs_1s_2\cos{({\phi}-{\psi})}
\end{align*} \\

From the identity (E.24) we get \\

\begin{align*}
&K({\alpha}+{\psi})H({\alpha}+{\phi})-K({\alpha}+{\phi})H({\alpha}+{\psi}) \\
=&4Q\sin{({\psi}-{\phi})}
\end{align*} \\

and then \\

\begin{align*}
s_1s_2\cos{({\phi}-{\psi})}=\sin{({\phi}-{\psi})}
\end{align*} \\

\end{proof}

\begin{remark}
Equations (47,63) together with equation 71 result in equation 65 \\
\end{remark}

\begin{proof}
\begin{align*}
\text{let } \hspace{0.5cm} &{\phi}-{\psi}={\delta}, \hspace{0.25cm} {\psi}={\phi}-{\delta} \\
\text{Then } \hspace{0.5cm} &K({\alpha}+{\psi}) = K({\alpha}+{\phi}-{\delta}) = K({\alpha}+{\phi})\cos{{\delta}}+M({\alpha}+{\phi})\sin{{\delta}} \\
\text{and } \hspace{0.5cm}
&2s_2s_3\cos{({\phi}-{\delta})} = \cos{{\delta}}K({\alpha}+{\phi})+\sin{{\delta}}M({\alpha}+{\phi}) \\
&2s_2s_3[\cos{{\phi}}+\sin{{\phi}}\tan{{\delta}}] = 2s_1s_3\sin{{\phi}}+\tan{{\delta}}M({\alpha}+{\phi}) \\
&2s_2s_3\cos{{\phi}}+2s_1s_3\sin{{\phi}}s_2^2 = 2s_1s_3\sin{{\phi}}+s_1s_2M({\alpha}+{\phi}) \\
&2s_2s_3\cos{{\phi}} - s_1s_2M({\alpha}+{\phi}) = 2s_1s_3\sin{{\phi}}[1-s_2^2] \\
&2\frac{s_3}{s_1}\cos{{\phi}}-M({\alpha}+{\phi}) = 4s_3\sin{{\phi}}u_2
\end{align*} \\
\end{proof}

\section{The equations for the rational leaning box} 

The unknowns are \\

\begin{align}
s_1,s_2,s_3,s_4 ; \hspace{0.25cm} 0<s_k<1 \\
u_k=\frac{1-s_k^2}{2s_k}, \hspace{0.25cm} Q=s_3s_4 
\end{align} \\

\begin{align}
\begin{split}
&\text{The parameters are ${\alpha}$ Heron angle with generator m, ${\psi}$ Euler} \\
&\text{angle and $u_1$} 
\end{split}
\end{align}  \\

Using Appendix E the equations (47-50) result in the relations \\

\begin{align}
K({\alpha})&=2s_3[s_2\cos^2{{\psi}}+\sin{{\psi}}\{2u_1\cos{{\psi}}-\frac{1}{s_2}\sin{{\psi}}\}] \\
N({\alpha})&=2s_4[-s_2\sin^2{{\psi}}+\cos{{\psi}}\{2u_1\sin{{\psi}}+\frac{1}{s_2}\cos{{\psi}}\}] \\
H({\alpha})&=2s_4[s_2\sin{{\psi}}\cos{{\psi}}+\sin{{\psi}}\{2u_1\sin{{\psi}}+\frac{1}{s_2}\cos{{\psi}}\}] \\
M({\alpha})&=2s_3[-s_2\sin{{\psi}}\cos{{\psi}}+\cos{{\psi}}\{2u_1\cos{{\psi}}-\frac{1}{s_2}\sin{{\psi}}\}]
\end{align} \\

These relations reproduce equations (29-30) \\

\begin{align}
s_3H({\alpha})+s_4M({\alpha})&=4Qu_1 \\
s_3N({\alpha})-s_4K({\alpha})&=4Qu_2
\end{align} \\

With \\

\begin{align}
{\lambda}=\tan{{\psi}} 
\end{align} \\

the independent equations for the rational leaning box become \\

\begin{align}
{\omega}_-({\alpha}) - {\lambda}{\omega}_+({\alpha}) &= s_2s_3+{\lambda}s_2s_4 \\
Q[{\omega}_+({\alpha}) + {\lambda}{\omega}_-({\alpha})] &= s_2s_3-{\lambda}s_2s_4 \\
2u_1[{\omega}_-({\alpha}) - {\lambda}{\omega}_+({\alpha})] + s_4-{\lambda}s_3 &= s_2[{\omega}_+({\alpha}) + {\lambda}{\omega}_-({\alpha})]
\end{align} \\

From equations (82-84) with \\

\begin{align}
{\lambda}=0 
\end{align} \\

the equations for the rational leaning box thus become \\

\begin{align}
s_2s_3&={\omega}_-({\alpha}) \\
s_2s_3&=Q{\omega}_+({\alpha}) \\
s_2{\omega}_+({\alpha})&=2u_1{\omega}_-({\alpha})+s_4
\end{align} \\
 
\begin{theorem}
For ${\lambda}=0$ the solutions of the equations for the rational leaning box are given by the
two rational parameters $s_1$, m, where \\

\begin{align}
0<s_1<1, \hspace{0.25cm} u_1=\frac{1-s_1^2}{2s_1} ,
\end{align}

and ${\alpha}$ has the generator m,

\begin{align}
0<{\alpha}<\frac{{\pi}}{4}, \hspace{0.25cm} 0<m<\sqrt{2}-1
\end{align}

as follows

\begin{align}
s_2&=2u_1\cot{(2{\alpha})} \\
s_3&=\frac{{\omega}_-({\alpha})}{s_2} \\
s_4&=\frac{s_2}{{\omega}_+({\alpha})} 
\end{align} \\

We also have to respect the inequality (D.13) \\

\end{theorem}

\begin{proof}
Equation 92 follows from equation 86. \\
Equation 87 reads \\

\begin{align*}
s_2=s_4{\omega}_+({\alpha})
\end{align*}

and gives equation 93. \\

From equation 88 we get \\

\begin{align*}
s_2{\omega}_+^2({\alpha}) = 2u_1{\omega}_-({\alpha}){\omega}_+({\alpha})+s_4{\omega}_+({\alpha})
\end{align*}

or 

\begin{align*}
s_2[{\omega}_+^2({\alpha})-1] = 2u_1{\omega}_-({\alpha}){\omega}_+({\alpha})
\end{align*} \\

resulting in 

\begin{align*}
s_2\sin{(2{\alpha})}=2u_1\cos{(2{\alpha})}
\end{align*} \\

Equation 90 assures that ${\omega}_-({\alpha}) > 0$ \\

\end{proof}

\begin{theorem}
For ${\lambda}=0$ the cuboid limit $s_4=s_3$ is impossible. Thus there is no perfect cuboid in this family.
\end{theorem}

\begin{proof}
From equation (92,93) we find 

\begin{align}
s_2^2s_3 = s_4\cos{(2{\alpha})}
\end{align}

Since 2${\alpha}$ is a Heron angle, $\cos{(2{\alpha})}$ can not be the square of a rational number, (Appendix A, Lemma 1) \\

\begin{remark}
For ${\psi}=0$, we get from page 4 that

\begin{align*}
{\omega}_+({\sigma}_1) &= \sqrt{2},& {\omega}_-({\sigma}_1) &= 0 \\
\end{align*} 

or

\begin{align*}
\cos{\sigma}_1 &= \frac{1}{\sqrt{2}},& \sin{\sigma}_1 &= \frac{1}{\sqrt{2}}
\end{align*} 

\begin{align*}
&\text{meaning that} \hspace{0.25cm} {\sigma}_1 = \frac{\pi}{4} \\
&\text{and thus} \hspace{0.25cm} {\alpha}+{\alpha}_1 = 2 {\sigma}_1 = \frac{\pi}{2}
\end{align*} 

\end{remark}

\end{proof}

\clearpage

\begin{example}
\begin{align*}
s_1=\frac{1}{2}, \hspace{0.25cm} m=\frac{1}{3} 
\end{align*} \\
\begin{align*}
\cos{{\alpha}}=\frac{4}{5}, \hspace{0.25cm} \sin{{\alpha}}=\frac{3}{5}, \hspace{0.25cm} \cos{(2{\alpha})}=\frac{7}{25}, \hspace{0.25cm} \sin{(2{\alpha})}=\frac{24}{25}, 
\hspace{0.25cm} {\omega}_+({\alpha})=\frac{7}{5}, \hspace{0.25cm} {\omega}_-({\alpha})=\frac{1}{5} \\
\end{align*} \\
Then 
\begin{align*}
s_2=\frac{7}{16}, \hspace{0.25cm} s_3=\frac{16}{35}, \hspace{0.25cm} s_4=\frac{5}{16} \\
\end{align*} \\
and \\
\begin{align*}
u_1&=\frac{840}{1120},& &u_2=\frac{1035}{1120},&&u_3=\frac{969}{1120},&&u_4=\frac{1617}{1120} \\
v_1&=\frac{1400}{1120},& &v_2=\frac{1525}{1120},&&v_3=\frac{1481}{1120},&&v_4=\frac{1967}{1120} 
\end{align*} \\

The leaning box is then given by \\

\begin{align*}
x&=1120,& &y=840,& z=1035 \\
a&=1400,& &b=1525& \\
c_1&=969,& &c_2=1617 \\
d_1&=1481,& &d_2=1967 \\
\end{align*} \\

\end{example}

\clearpage

\begin{example}

\begin{align*}
s_1=\frac{12}{25}, \hspace{0.25cm} m=\frac{1}{3}
\end{align*} \\

\begin{align*}
\cos{(2{\alpha})}=\frac{7}{25}, \hspace{0.25cm} \sin{(2{\alpha})}=\frac{24}{25}, \hspace{0.25cm} {\omega}_+({\alpha})=\frac{7}{5}, \hspace{0.25cm} {\omega}_-({\alpha})=\frac{1}{5} \\
\end{align*} \\

Then

\begin{align*}
s_2=\frac{3367}{7200}, \hspace{0.25cm} s_3=\frac{1440}{3367}, \hspace{0.25cm} s_4=\frac{2405}{7200}
\end{align*} \\

and

\begin{align*}
u_1&=\frac{38868648}{48484800},& &u_2=\frac{40503311}{48484800},&&u_3=\frac{46315445}{48484800},&&u_4=\frac{64478365}{48484800} \\
v_1&=\frac{62141352}{48484800},& &v_2=\frac{63176689}{48484800},&&v_3=\frac{67051445}{48484800},&&v_4=\frac{80673635}{48484800} 
\end{align*} \\

The leaning box is then given by \\

\begin{align*}
x&=48484800,& &y=38868648, &&z=40503311 \\
a&=62141352,& &b=63176689 \\
c_1&=46315445,& &c_2=64478365 \\
d_1&=67051445,& &d_2=80673635 \\
\end{align*} \\

\end{example}

\clearpage
\section{Symmetry and the general equations} 

We have solved the general equations (82-84) for \\

\begin{align*}
{\lambda} = \tan {\psi} = 0 \\
\end{align*}

For a given $s_1$, we found $s_2, s_3, s_4$ in terms of $u_1$ and ${\alpha}$ (91-93). \\
The two parallelograms I, II in (16-17) are described by the two Heron angles ${\alpha}$ and ${\alpha}_1$ and according to the observation in Appendix D are also described by the two Heron angles ${\beta}$ and ${\beta}_1$. Their generators are given by \\

\begin{align}
m({\alpha}) = \frac{2u_2 + u_3 - u_4}{2u_1 + u_3 +u_4},\hspace{0.25cm} m({\alpha}_1) = \frac{2v_2 + v_3 - v_4}{2u_1 + v_3 +v_4}
\end{align} \\

\begin{align}
m({\beta}) = \frac{2u_2 - u_3 + u_4}{2u_1 + u_3 +u_4},\hspace{0.25cm} m({\beta}_1) = \frac{2v_2 - v_3 + v_4}{2u_1 + v_3 +v_4}
\end{align} \\

Now the equation (16) is invariant under the interchange of $u_3$ and $u_4$, and the equation (17) is invariant under the interchange of $v_3$ and $v_4$. According to (13-14) this means that the equations (16-17) are invariant under the interchange of $s_3$ with $s_4$. \\

From (p.4) and (81) we find 

\begin{align*}
{\alpha} + {\alpha}_1 = 2 {\sigma}_1 
\end{align*} 

\begin{align*}
{\omega}_+({\sigma}_1) = \sqrt{2} \cos{\psi},\hspace{0.25cm} {\omega}_-({\sigma}_1) = \sqrt{2} \sin{\psi},\hspace{0.25cm} {\lambda} = \tan{{\psi}} = {\frac{{\omega}_-({\sigma}_1)}{{\omega}_+({\sigma}_1)}}
\end{align*} \\

let k be the generator of ${\alpha} + {\alpha}_1$ \\

\begin{align}
k = \frac{m({\alpha}) + m({\alpha}_1)}{1 - m({\alpha}) m({\alpha}_1)} 
\end{align} \\

Then
\begin{align}
\sin{(2{\sigma}_1)} = \frac{2k}{1 + k^2},\hspace{0.25cm} \cos{(2{\sigma}_1)} = \frac{1 - k^2}{1 + k^2}
\end{align}

and 

\begin{align}
{\lambda} = tan{{\psi}} = \frac{{\omega}_-({\sigma}_1)}{{\omega}_+({\sigma}_1)} = \frac{{\omega}_-^2({\sigma}_1)}{{\omega}_+({\sigma}_1){\omega}_-({\sigma}_1)} = \frac{1 - \sin{(2{\sigma}_1)}}{\cos{(2{\sigma}_1)}}
\end{align} \\

resulting in

\begin{align}
{\lambda} = \frac{1 - k}{1 + k}
\end{align} \\

Let $\bar{k}$ be the generator of ${\beta} + {\beta}_1$ \\

\begin{align}
\bar{k} = \frac{m({\beta}) + m({\beta}_1)}{1 - m({\beta}) m({\beta}_1)}
\end{align}

and thus 

\begin{align}
\bar{{\lambda}} = \frac{1 - \bar{k}}{1 + \bar{k}}
\end{align}  \\

Therefore the parameters $u_1, {\beta}, \bar{{\lambda}}$ also satisfy the general equations, however with the interchange of $s_3$ with $s_4$. \\

\clearpage
\section{An equivalent set of equations for the rational leaning box} 

For the three parallelograms we use

\begin{align*}
&\text{I. }  &&2u_1^2+2u_2^2= u_3^2+u_4^2 \\
&\text{II. } &&2u_1^2+2v_2^2= v_3^2+v_4^2 \\
&\text{III. }&&v_2^2-u_2^2 = 1, v_3^2-u_3^2 = 1, v_4^2-u_4^2 = 1
\end{align*}

where $u_k = \cot{{\psi}_k}$, $v_k=\frac{1}{\sin{{\psi}_k}}$, ${\psi}_k$ Heron angles with \\
generators $s_k \in (0,1)$, k=1,2,3,4. \\

We give $s_1$ resulting in giving $u_1$ and $v_1$. \\

According to Appendix D we have the following parameterization \\

\begin{align}
&\text{I.  parameters} & &u_1, {\alpha}, {\beta}; {\alpha}, {\beta} \in (0,\frac{\pi}{2}) \\
&                      & &{\alpha}+{\beta}=2{\sigma}, {\alpha}-{\beta}=2{\delta}; {\alpha}={\sigma}+{\delta}, {\beta}={\sigma}-{\delta} \\
&\text{II. parameters} & &u_1, {\alpha_1}, {\beta_1}; {\alpha_1}, {\beta_1} \in (0,\frac{\pi}{2}) \\
&                      & &{\alpha}_1+{\beta}_1=2{\sigma}_1, {\alpha}_1-{\beta}_1=2{\delta}_1; {\alpha}_1={\sigma}_1+{\delta}_1, {\beta}_1={\sigma}_1-{\delta}_1
\end{align}

and thus the representation

\begin{align}
&u_2 = u_1{\tan{\sigma}},   & &u_3=u_1{\frac{{\omega}_+({\delta})}{\cos{\sigma}}},     & &u_4=u_1{\frac{{\omega}_-({\delta})}{\cos{\sigma}}} \\
&v_2 = u_1{\tan{\sigma}_1}, & &v_3=u_1{\frac{{\omega}_+({\delta}_1)}{\cos{\sigma}_1}}, & &v_4=u_1{\frac{{\omega}_-({\delta}_1)}{\cos{\sigma}_1}}
\end{align}

conversely

\begin{align}
&\tan{{\sigma}}=\frac{u_2}{u_1},  & &\tan{{\delta}}=\frac{u_3-u_4}{u_3+u_4} \\
&\tan{{\sigma}_1}=\frac{v_2}{u_1},& &\tan{{\delta}_1}=\frac{v_3-v_4}{v_3+v_4}
\end{align}

Observe that ${\psi}_1$ is a Heron angle and ${\sigma}, {\delta}, {\sigma}_1, {\delta}_1$ are Euler angles. We rename ${\psi}_1={\psi}$ \\

We now have the conditions

\begin{align}
u_1^2[{\tan^2{\sigma}_1}-{\tan^2{\sigma}}] &= u_1^2[\frac{1}{{\cos^2{\sigma}_1}}-\frac{1}{{\cos^2{\sigma}}}]=1 \\
u_1^2[\frac{{{\omega}_+}^2({\delta}_1)}{{\cos^2{\sigma}_1}} - \frac{{{\omega}_+}^2({\delta})}{{\cos^2{\sigma}}}] &= u_1^2[\frac{1+\sin{(2{\delta}_1)}}{{\cos^2{\sigma}_1}} - \frac{1+\sin{(2{\delta})}}{{\cos^2{\sigma}}}]=1 \\
u_1^2[\frac{{{\omega}_-}^2({\delta}_1)}{{\cos^2{\sigma}_1}} - \frac{{{\omega}_-}^2({\delta})}{{\cos^2{\sigma}}}] &= u_1^2[\frac{1-\sin{(2{\delta}_1)}}{{\cos^2{\sigma}_1}} - \frac{1-\sin{(2{\delta})}}{{\cos^2{\sigma}}}]=1
\end{align}

or using eq.(111), the conditions reduce to 

\begin{align}
&{\tan^2{\sigma}_1} - {\tan^2{\sigma}} = {\tan^2{\psi}} \\
&\frac{\sin(2{\delta})}{{\cos^2{\sigma}}} = \frac{\sin(2{\delta}_1)}{{\cos^2{{\sigma}_1}}}
\end{align}

Using the generators, we introduce the abbreviations

\begin{align}
M &= m({\alpha}+{\beta}) = m(2{\sigma}) = \tan{\sigma} > 0 \\
M_1 &= m({\alpha}_1+{\beta}_1) = m(2{\sigma}_1) = \tan{\sigma}_1 > 0 \\
N &= m({\alpha}-{\beta}) = m(2{\delta}) = \tan{\delta} \\
N_1 &= m({\alpha}_1-{\beta}_1) = m(2{\delta}_1) = \tan{\delta}_1
\end{align}

Then

\begin{align}
\tan{\alpha} &= \tan{({\sigma}+{\delta})} = \frac{\tan{\sigma} + \tan{\delta}}{1-\tan{\sigma}\tan{\delta}} = \frac{M+N}{1-MN} \\
\tan{\alpha}_1 &= \tan{({\sigma}_1+{\delta}_1)} = \frac{\tan{\sigma}_1 + \tan{\delta}_1}{1-\tan{\sigma}_1\tan{\delta}_1} = \frac{M_1+N_1}{1-M_1N_1}
\end{align}

resulting in

\begin{align}
N &= \frac{\tan{\alpha} - M}{1+M \tan{\alpha}} \\
N_1 &= \frac{\tan{\alpha}_1 - M_1}{1+M_1 \tan{\alpha}_1}
\end{align}

The conditions (114) (115) then become

\begin{align}
M_1^2-M^2 &= \tan^2{{\psi}} \\
\frac{2N}{1+N^2}(1+M^2) &= \frac{2N_1}{1+N_1^2}(1+M_1^2)
\end{align}

Or, using eqs (122), (123) we find

\begin{align}
(1-M^2) \sin(2{\alpha})-2M\cos(2{\alpha}) = (1-M_1^2)\sin(2{\alpha}_1)-2M_1\cos(2{\alpha}_1)
\end{align}

Now

\begin{align}
\cos{\delta}\cos{\sigma} &= \frac{1}{2}[\cos({\delta}-{\sigma}) + \cos({\delta}+{\sigma})] = \frac{1}{2}[\cos{\beta} + \cos{\alpha}] \\
\sin{\delta}\cos{\sigma} &= \frac{1}{2}[\sin({\delta}-{\sigma}) + \sin({\delta}+{\sigma})] = \frac{1}{2}[\sin{\alpha} - \sin{\beta}] \\
\frac{1}{{\cos^2{\sigma}}} &= 1+{\tan^2{\sigma}} = 1+M^2
\end{align}
 
gives the following representation

\begin{align}
\frac{{\omega}_+({\delta})}{\cos{\sigma}} = \frac{\cos{\delta} + \sin{\delta}}{{\cos^2{\sigma}}} \cos{\sigma} = \frac{1}{2}(1+M^2)[{\omega}_+({\alpha}) + {\omega}_-({\beta})] \\
\frac{{\omega}_-({\delta})}{\cos{\sigma}} = \frac{\cos{\delta} - \sin{\delta}}{{\cos^2{\sigma}}} \cos{\sigma} = \frac{1}{2}(1+M^2)[{\omega}_-({\alpha}) + {\omega}_+({\beta})]
\end{align}

From

\begin{align}
\cos{\beta} &= \cos({\alpha} + {\beta} - {\alpha}) = \cos({\alpha}+{\beta}) \cos{\alpha} + \sin({\alpha} + {\beta}) \sin{\alpha} \\
\cos{\beta} &= \frac{1-M^2}{1+M^2} \cos{\alpha} + \frac{2M}{1+M^2} \sin{\alpha} \\
\sin{\beta} &= \sin({\alpha} + {\beta} - {\alpha}) = \sin({\alpha}+{\beta}) \cos{\alpha} - \cos({\alpha} + {\beta}) \sin{\alpha} \\
\sin{\beta} &= \frac{2M}{1+M^2} \cos{\alpha} - \frac{1-M^2}{1+M^2} \sin{\alpha} 
\end{align}

we find

\begin{align}
{\omega}_+({\beta}) = \frac{1-M^2}{1+M^2} {{\omega}_-({\alpha})} + \frac{2M}{1+M^2} {\omega}_+({\alpha}) \\
{\omega}_-({\beta}) = \frac{1-M^2}{1+M^2} {{\omega}_+({\alpha})} - \frac{2M}{1+M^2} {\omega}_-({\alpha})
\end{align}

and then

\begin{align}
{\omega}_+({\alpha}) + {\omega}_-({\beta}) = \frac{2}{1+M^2}[{\omega}_+({\alpha}) - M{\omega}_-({\alpha})] \\
{\omega}_-({\alpha}) + {\omega}_+({\beta}) = \frac{2}{1+M^2}[{\omega}_-({\alpha}) + M{\omega}_+({\alpha})]
\end{align}

Finally, from eq.(107)

\begin{align}
&u_2 = u_1M \\
&u_3 = u_1[{\omega}_+({\alpha}) - M{\omega}_-({\alpha})] \\
&u_4 = u_1[{\omega}_-({\alpha}) + M{\omega}_+({\alpha})]
\end{align}

and from eq.(108)

\begin{align}
&v_2 = u_1M_1 \\
&v_3 = u_1[{\omega}_+({\alpha}_1) - M_1{\omega}_-({\alpha}_1)] \\
&v_4 = u_1[{\omega}_-({\alpha}_1) + M_1{\omega}_+({\alpha}_1)]
\end{align}

From the symmetry of interchanging ${\alpha}$ and ${\beta}$, which corresponds to the interchange of $u_3$ and $u_4$, we also have the representation

\begin{align}
&u_3 = u_1[{\omega}_-({\beta}) + M{\omega}_+({\beta})] \\
&u_4 = u_1[{\omega}_+({\beta}) - M{\omega}_-({\beta})]
\end{align}
 
\subsubsection*{\underline{Example}}
For the special case of ${\alpha}+{\alpha}_1=\frac{{\pi}}{2}$, resulting in 

\begin{align*}
\cos(2{\alpha}_1) = -\cos(2{\alpha}), \sin(2{\alpha}_1) = \sin(2{\alpha})
\end{align*}

and given the two Heron angles ${\psi}, {\alpha}$ ; $u_1=\cot{\psi}$, \\
the conditions (124) (126) read

\begin{align*}
M_1-M = 2\cot(2{\alpha}), M_1+M = \frac{1}{2}{\tan^2{\psi}}\tan(2{\alpha})
\end{align*}

or

\begin{align*}
M=\frac{1}{4}[{\tan^2{\psi}}\tan(2{\alpha}) - 4\cot(2{\alpha})] \\
M_1=\frac{1}{4}[4\cot(2{\alpha}) + {\tan^2{\psi}}\tan(2{\alpha})]
\end{align*}

Observe that this two-parameter family of rational leaning boxes has no cuboid limit, because in the cuboid limit N=0, $N_1=0$, implying

\begin{align*}
M=\tan{\alpha}, M_1=\tan{\alpha}_1 = \frac{1}{\tan{\alpha}}
\end{align*}

Then

\begin{align*}
M_1^2 - M^2 = {\tan^2{\psi}}
\end{align*}

reads

\begin{align*}
1-{\tan^4{\alpha}} = \tan^2{\alpha}\tan^2{\psi}
\end{align*}

which has no rational solutions [6]. \\

We now split the eq.(126) into two parts

\begin{align}
&(M^2-1)\sin(2{\alpha}) + 2M\cos(2{\alpha}) = 4D \\
&(M_1^2-1)\sin(2{\alpha}_1) + 2M_1\cos(2{\alpha}_1) = 4D
\end{align}

According to the Appendix F, they have the following parameter representations, replacing $\lambda$ by $r$, respectively by $r_1$.

\begin{align}
M &= 2r + \tan{\alpha} &D=r[1+r\sin(2{\alpha})] \\
M_1 &= 2r_1 + \tan{\alpha}_1 &D=r_1[1+r_1\sin(2{\alpha}_1)]
\end{align}

Now, using [8], the equation

\begin{align}
r[1+r\sin(2{\alpha})] = r_1[1+r_1\sin(2{\alpha}_1)]
\end{align}

has the parameter representation

\begin{align}
r=f\frac{1+f\sin(2{\alpha}_1)}{1-f^2\sin(2{\alpha})\sin(2{\alpha}_1)} \\
r_1=f\frac{1+f\sin(2{\alpha})}{1-f^2\sin(2{\alpha})\sin(2{\alpha}_1)}
\end{align}

Conversely, for $f \neq 0$ \\

case (i)

\begin{align}
r_1 \neq r, f=\frac{r_1-r}{r\sin(2{\alpha})-r_1\sin(2{\alpha}_1)}
\end{align}

case (ii)

\begin{align}
r_1 = r, f=\frac{r}{1+r\sin(2{\alpha})}
\end{align}

\begin{align*}
&\text{and} &  &\sin(2{\alpha}_1) = \sin(2{\alpha}) \\
&\text{or}  &  &\cos({\alpha}_1+{\alpha})\sin({\alpha}_1-{\alpha}) = 0
\end{align*}
\begin{align*}
&\text{For the non-trivial case,} {\alpha} \neq {\alpha}_1 \text{we have} \\
&{\alpha}_1 + {\alpha} = \frac{{\pi}}{2}\text{, which is the example above}
\end{align*}

This parameter representation can be verified directly.

\section{Cuboid Limit} 

The cuboid limit is given by \\

\begin{align}
{\beta} \rightarrow {\alpha} \hspace{0.25cm} \text{and} \hspace{0.25cm} {\beta}_1 \rightarrow {\alpha}_1
\end{align}

From (118, 119) this means

\begin{align}
N \rightarrow 0 \hspace{0.25cm} \text{and} \hspace{0.25cm} N_1 \rightarrow 0
\end{align}

and from (122, 123) we get

\begin{align}
M \rightarrow \tan{\alpha} \hspace{0.25cm} \text{and} \hspace{0.25cm} M_1 \rightarrow \tan{\alpha}_1
\end{align}

or equivalently from (150,151) that  

\begin{align}
r \rightarrow 0 \hspace{0.25cm} \text{and} \hspace{0.25cm} r_1 \rightarrow 0
\end{align}

Looking at the equation (152) and letting $ \frac{r_1}{r} = s $, we have two cases

\begin{equation}
\begin{aligned}
\text{(i)} \hspace{0.25cm} s=1 \hspace{0.25cm} \text{implies} \hspace{0.25cm} &\sin{(2{\alpha})}=\sin{(2{\alpha}_1)} \hspace{0.25cm} \text{or} \hspace{0.25cm} \\
{\alpha}+{\alpha}_1 = &\frac{\pi}{2}
\end{aligned}
\end{equation}

which, according to the example on page 21 has no cuboid limit. This is also the family found in section 5.

\begin{equation}
\begin{aligned}
\text{(ii)} \hspace{0.25cm} s \neq 1 \hspace{0.25cm} \text{. Let}& \\
F(r,r_1) &= \frac{r_1}{r} - \frac{1+r\sin(2{\alpha})}{1+r_1\sin(2{\alpha_1})} 
\end{aligned}
\end{equation}

For the cuboid limit

\begin{align}
F(r,r_1) = s-1 \neq 0 \text{,}
\end{align}

which contradicts equation (152).\\

Thus there is also no cuboid limit. \\

In conclusion, there is no perfect cuboid.

\clearpage
\section{Consequence} 

\subsubsection*{\underline{Corollary}}
Let ${\alpha}_1, {\alpha}$ be Euler angles and ${\psi}$ be a Heron angle. Then in the equation

\begin{align}
{\tan^2{\alpha}_1} - {\tan^2{\alpha}} = {\tan^2{\psi}}
\end{align}

not both ${\alpha}$ and ${\alpha}_1$ can be Heron angles

\begin{proof}
In the cuboid limit only equation (124) is the surviving condition. If both ${\alpha}$ and ${\alpha}_1$ were Heron angles we would have a perfect cuboid. This is a contradiction.
\end{proof}

\subsubsection*{\underline{Example}}
Euler Cuboid (body diagonal not rational)

\begin{align*}
\tan{\alpha}_1 &= \frac{125}{240}, &\cos{\alpha}_1 &= \frac{48}{\sqrt{2929}}, &\sin{\alpha}_1 &= \frac{25}{\sqrt{2929}} && \\
\tan{\alpha} &= \frac{44}{240}, &\cos{\alpha} &= \frac{60}{61}, &\sin{\alpha} &= \frac{11}{61}, &m({\alpha}) &= \frac{1}{11} \\
\tan{\psi} &= \frac{117}{240}, &\cos{\psi} &= \frac{80}{89}, &\sin{\psi} &= \frac{39}{89}, &m({\psi}) &= \frac{3}{13}
\end{align*}

\subsubsection*{\underline{Example}}
Face Cuboid (one face diagonal not rational)

\begin{align*}
\tan{\alpha}_1 &= \frac{765}{520}, &\cos{\alpha}_1 &= \frac{104}{185}, &\sin{\alpha}_1 &= \frac{153}{185}, &m({\alpha}_1) &= \frac{9}{17} \\
\tan{\alpha} &= \frac{756}{520}, &\cos{\alpha} &= \frac{130}{\sqrt{52621}}, &\sin{\alpha} &= \frac{189}{\sqrt{52621}} &&\\
\tan{\psi} &= \frac{117}{520}, &\cos{\psi} &= \frac{40}{41}, &\sin{\psi} &= \frac{9}{41}, &m({\psi}) &= \frac{1}{9}
\end{align*}

\clearpage
\section*{Appendix A}
\subsection*{\centerline{Generator of an angle}\\ \centerline{Heron angle, Euler angle} \\}
\subsubsection*{\underline{Definition 1}}
For an arbitrary angle $\alpha$, its generator is defined by\\ 
\\

\setcounter{equation}{0}
\renewcommand{\theequation}{A.\arabic{equation}}

\begin{equation}
\begin{aligned}
m({\alpha})=\frac{\sin{\alpha}}{1+\cos{\alpha}}=\tan(\frac{\alpha}{2})\\
\end{aligned}
\end{equation}

Consequently

\begin{equation}
\begin{aligned}
\cos {\alpha}  = \frac{1-m^2({\alpha})}{1+m^2({\alpha})}, \hspace{0.25cm} \sin {\alpha} = \frac{2m(\alpha)}{1+m^2(\alpha)}\\
\end{aligned}
\end{equation}
\\

We have the following properties

\begin{align}
m(-{\alpha})  &= -m(\alpha)\\
m^2({\alpha})  &= \frac{1-\cos{\alpha}}{1+\cos{\alpha}}\\
m({\alpha}+{\beta})  &= \frac{m({\alpha})+m({\beta})}{1-m({\alpha})m({\beta})}\\
\frac{dm({\alpha})}{d{\alpha}} &= \frac{1}{2} \frac{1}{\cos^2(\frac{\alpha}{2})} > 0
\end{align}

Then m($\alpha$) is an increasing function of $\alpha$, with 

\begin{equation*}
\begin{aligned}
m(0) = 0, \hspace{0.25cm} m(\frac{{\pi}}{2}) = 1\\ 
\end{aligned}
\end{equation*}\\

For $0<{\alpha}<\frac{{\pi}}{2}$ the generator satisfies

\begin{equation}
\begin{aligned}
0<m({\alpha})<1\\
\end{aligned}
\end{equation}\\

\subsubsection*{\underline{Definition 2}}
An angle ${\alpha}$ is called a Heron angle if both $\sin{\alpha}$ and $\cos{\alpha}$ are rational.\\

The generator of a Heron angle is rational and vise versa. The sum and the difference of two Heron angles are Heron angles.\\

The complement 

\begin{equation}
\begin{aligned}
\bar{{\alpha}} = \frac{{\pi}}{2} - {\alpha} \\
\end{aligned}
\end{equation}\\

of a Heron angle ${\alpha}$ is a Heron angle. \\

\subsubsection*{\underline{Definition 3}}
An angle ${\alpha}$ is called an Euler angle if $\tan{\alpha}$ is rational.\\

If ${\alpha}$ is an Euler angle then $2{\alpha}$ is a Heron angle.
A Heron angle is an Euler angle.\\

\subsubsection*{\underline{Lemma 1}}
Let ${\alpha}$ be a Heron angle, $0\leq{\alpha}\leq{\pi}$.\\
Then the equation 

\begin{equation}
\begin{aligned}
\sin{{\alpha}} = {\lambda}^2 \\
\end{aligned}
\end{equation}\\

where ${\lambda}$ is a rational number, $0\leq{\lambda}\leq1$, has only the trivial \\
solutions ${\alpha}=0, \hspace{0.25cm} {\alpha}=\frac{{\pi}}{2}, \hspace{0.25cm} {\alpha}={\pi}$\\ 

\begin{proof}
let ${\lambda}=\frac{a}{b}, \hspace{0.25cm} 0<a<b, \hspace{0.25cm}$ a and b integers.\\

Then $\sin{\alpha}$ is rational and
 
\begin{equation}
\begin{aligned}
\cos^2{{\alpha}} = 1-\sin^2{{\alpha}}=1-{\lambda}^4=\frac{b^4-a^4}{b^4} \\
\end{aligned}
\end{equation}\\

But according to Euler \cite{nagell}, $b^4-a^4$ can not be the square of an integer. Thus $\cos {\alpha}$ is not rational, i.e. ${\alpha}$ is not a Heron angle, except for the trivial cases. 
\end{proof}

\subsubsection*{\underline{Lemma 2}}
Let ${\alpha}$ be a Heron angle, $0\leq{\alpha}\leq\frac{{\pi}}{2}$.\\

Then the equation

\begin{equation}
\begin{aligned}
\tan{{\alpha}} = {\lambda}^2 \\
\end{aligned}
\end{equation}\\

where ${\lambda}$ is a rational number, ${\lambda}\geq0$, has only the trivial solution ${\alpha}=0$.\\

\begin{proof}
Let ${\lambda}=\frac{a}{b},\hspace{0.25cm} a>0, \hspace{0.25cm} b>0, \hspace{0.25cm}$ a and b integers.\\

Then $\tan{{\alpha}}$ is rational and\\ 

\begin{equation}
\begin{aligned}
\frac{1}{\cos^2{\alpha}}=1+\tan^2{{\alpha}}=1+{\lambda}^4=\frac{b^4+a^4}{b^4}\\
\end{aligned}
\end{equation}\\

But according to Euler \cite{nagell}, $b^4+a^4$ is not the square of an integer. Thus $\cos{\alpha}$ is not rational, except for the trivial case.\\ 

Thus in any case, for a Heron angle ${\alpha}, \sin{{\alpha}}, \cos{{\alpha}}, \tan{{\alpha}}, \cot{{\alpha}}$ can not be the square of a rational number, except for the trivial cases.\\
\end{proof}

\subsubsection*{\underline{Corollary 1}}
The elliptic curve\\

\begin{equation}
\begin{aligned}
y^2=x(1-x^2)
\end{aligned}
\end{equation}\\

has only the trivial rational points (x,y), namely (-1,0),(0,0),(1,0) \\

\begin{proof}
Let ${\alpha}$ be a Heron angle and\\
\begin{equation*}
\begin{aligned}
{\lambda}^2=\sin{2{\alpha}}\\
\end{aligned}
\end{equation*}\\

Then ${\lambda}$ can not be rational, except for the trivial cases.\\
Let x be the generator of ${\alpha}$\\

\begin{align*}
{\lambda}^2&=2\sin{{\alpha}}\cos{{\alpha}}=2\frac{2x}{1+x^2}\frac{1-x^2}{1+x^2}\\
{\lambda}^2&=[\frac{2}{1+x^2}]^2x(1-x^2)\\
\end{align*} \\

Then\\

\begin{equation}
\begin{aligned}
y^2=x(1-x^2)\\
\end{aligned}
\end{equation}\\

has only the trivial rational points.\\
\end{proof}

\subsubsection*{\underline{Corollary 2}}
The elliptic curve\\

\begin{equation}
\begin{aligned}
y^2=2x(1-x^2)\\
\end{aligned}
\end{equation}\\

has only the trivial rational points (x,y), namely (-1,0),(0,0),(1,0) \\

\begin{proof}
Let ${\alpha}$ be a Heron angle with generator x, and\\ 

\begin{equation*}
\begin{aligned}
{\lambda}^2=\tan{{\alpha}}\\
\end{aligned}
\end{equation*}\\

Then ${\lambda}$ can not be rational, except for the trivial case.\\

Now 

\begin{equation*}
\begin{aligned}
{\lambda}^2=\frac{2x}{1-x^2}=[\frac{1}{1-x^2}]^22x(1-x^2)\\
\end{aligned}
\end{equation*}

Then\\

\begin{equation*}
\begin{aligned}
y^2=2x(1-x^2)\\
\end{aligned}
\end{equation*}\\

has only the trivial rational points.
\end{proof}

\clearpage
\section*{Appendix B} 
\subsection*{\centerline{Rotations}}

\setcounter{equation}{0}
\renewcommand{\theequation}{B.\arabic{equation}}

A rotation in two dimensions is given by the matrix

\begin{equation}
\begin{aligned}
R({\alpha})=
\left( \begin{array}{cc}
\cos{\alpha} & -\sin{\alpha} \\
\sin{\alpha} & \cos{\alpha}
\end{array} \right)
\end{aligned}
\end{equation}\\

These form an Abelian group.

\begin{align}
&\text{Group multiplication}	\hspace{2cm} 	&&R({\alpha})R({\beta})=R({\alpha}+{\beta})\\
&\text{Identity}		&&R(0)\\
&\text{Inverse}			&&R({\alpha})^{-1}=R(-{\alpha})
\end{align}\\

For two two-dimensional vectors, related by a rotation, we have 

\begin{align}
\binom{x_1}{y_1} = R({\alpha}) \binom{x_2}{y_2}
\end{align}\\

They have the same length, i.e.

\begin{align}
x_1^2+y_1^2 = x_2^2+y_2^2
\end{align}\\

Conversely we get the rotation angle ${\alpha}$ through

\begin{align}
\cos{\alpha}=\frac{1}{x_2^2+y_2^2}[x_1x_2+y_1y_2]\\
\sin{\alpha}=\frac{1}{x_2^2+y_2^2}[y_1x_2-x_1y_2]
\end{align}

\clearpage
\section*{Appendix C}
\subsection*{\centerline{The ${\omega}$-functions}}

\setcounter{equation}{0}
\renewcommand{\theequation}{C.\arabic{equation}}

\subsubsection*{\underline{Definition}}
For an angle ${\alpha}$ we introduce the ${\omega}$-functions by

\begin{align}
{\omega}_+({\alpha})=\cos{\alpha}+\sin{\alpha}, \hspace{0.25cm} {\omega}_-({\alpha})=\cos{\alpha}-\sin{\alpha}
\end{align}

We then have the following properties

\begin{align}
{\omega}_+({-\alpha})&={\omega}_-({\alpha})\\
{\omega}_-({-\alpha})&={\omega}_+({\alpha})\\
{\omega}_+^2({\alpha})+{\omega}_-^2({\alpha})&=2\\
{\omega}_+({\alpha}){\omega}_-({\alpha})&=\cos(2{\alpha})\\
{\omega}_+^2({\alpha})&=1+\sin(2{\alpha})\\
{\omega}_-^2({\alpha})&=1-\sin(2{\alpha})
\end{align}

For two angles ${\alpha}$ and ${\beta}$ we introduce 

\begin{align}
{\alpha}+{\beta}=2{\sigma},& \hspace{0.25cm} {\alpha}-{\beta}=2{\delta}\\
{\alpha}={\sigma}+{\delta},& \hspace{0.25cm} {\beta}={\sigma}-{\delta}
\end{align}

We then find 

\begin{align}
{\omega}_+({\alpha}){\omega}_+({\beta})=\cos(2{\delta})+\sin(2{\sigma})\\
{\omega}_-({\alpha}){\omega}_-({\beta})=\cos(2{\delta})-\sin(2{\sigma})\\
{\omega}_+({\alpha}){\omega}_-({\beta})=\cos(2{\sigma})+\sin(2{\delta})
\end{align}

We also have the relation

\begin{align}
R({\beta}) \binom{{\omega}_+({\alpha}+{\beta})}{{\omega}_-({\alpha}+{\beta}}=\binom{{\omega}_+({\alpha})}{{\omega}_-({\alpha})}
\end{align}

or explicitly

\begin{align}
{\omega}_+({\alpha}+{\beta})&=\cos{\beta} {\omega}_+({\alpha}) + \sin{\beta} {\omega}_-({\alpha})\\
{\omega}_+({\alpha}+{\beta})&=\cos{\alpha} {\omega}_+({\beta}) + \sin{\alpha} {\omega}_-({\beta})\\
{\omega}_-({\alpha}+{\beta})&=-\sin{\beta} {\omega}_+({\alpha}) + \cos{\beta} {\omega}_-({\alpha})\\
{\omega}_-({\alpha}+{\beta})&=-\sin{\alpha} {\omega}_+({\beta}) + \cos{\alpha} {\omega}_-({\beta})
\end{align}

This gives the following relations

\begin{align}
{\omega}_+({\alpha})+{\omega}_+({\beta}) &= 2\cos{\delta} {\omega}_+({\sigma})\\
{\omega}_+({\alpha})-{\omega}_+({\beta}) &= 2\sin{\delta} {\omega}_-({\sigma})\\
{\omega}_-({\alpha})+{\omega}_-({\beta}) &= 2\cos{\delta} {\omega}_-({\sigma})\\
{\omega}_-({\alpha})-{\omega}_-({\beta}) &= -2\sin{\delta} {\omega}_+({\sigma})\\
{\omega}_+({\alpha})+{\omega}_-({\beta}) &= 2\cos{\sigma} {\omega}_+({\delta})\\
{\omega}_+({\alpha})-{\omega}_-({\beta}) &= 2\sin{\sigma} {\omega}_-({\delta})\\
{\omega}_-({\alpha})+{\omega}_+({\beta}) &= 2\cos{\sigma} {\omega}_-({\delta})\\
{\omega}_-({\alpha})-{\omega}_+({\beta}) &= -2\sin{\sigma} {\omega}_+({\delta})
\end{align}

and finally \\

\begin{align}
\left(
    \begin{array}{c}
      {\omega}_-({\alpha}) \\
      {\omega}_+({\alpha})
    \end{array}
\right) = R(2{\alpha})
\left(
    \begin{array}{cc}
      {\omega}_+({\alpha}) \\
      {\omega}_-({\alpha}) 
    \end{array}
\right)
\end{align}

\clearpage
\section*{Appendix D}
\subsection*{\centerline{The Rational Parallelogram}}

\setcounter{equation}{0}
\renewcommand{\theequation}{D.\arabic{equation}}

A parallelogram with its sides $u_1, u_2$ and diagonals $u_3, u_4$ being positive rational numbers is called a rational parallelogram It is governed by the parallelogram equation

\begin{align}
2u_1^2+2u_2^2=u_3^2+u_4^2
\end{align} \\

In \cite{wyss12, wyss14} we found a bijective parameter representation for all rational parallelograms.\\
It is given by the rational scaling parameter\\

\begin{align}
u>0
\end{align} \\

and two rational parameters m,n \\

\begin{align}
0<m<1, \hspace{0.25cm} 0<n<1
\end{align} \\

The representation is given by \\

\begin{align}
u_1&=(1-mn)u \\
u_2&=(m+n)u \\
u_3&=(1+mn-n+m)u \\
u_4&=(1+mn+n-m)u
\end{align}\\

Conversely 

\begin{align}
4u=2u_1&+u_3+u_4 \\
m=\frac{2u_2+u_3-u_4}{4u}&, \hspace{0.25cm} n=\frac{2u_2-u_3+u_4}{4u}
\end{align} \\

Special cases: 

\begin{align}
\text{Rectangle:} \hspace{1cm} u_4&=u_3; \hspace{0.25cm} n=m \\
\text{Rhomboid:}  \hspace{1cm} u_2&=u_1; \hspace{0.25cm} n=\frac{1-m}{1+m}
\end{align} \\

From (D.4) and (D.5) we find \\

\begin{align}
(m+n)u_1 = (1-mn)u_2
\end{align}

and then \\

\begin{align}
n&=\frac{u_2-mu_1}{u_1+mu_2}, \hspace{0.25cm} 0<n<1 \\
m&=\frac{u_2-nu_1}{u_1+nu_2}, \hspace{0.25cm} 0<m<1
\end{align} \\

From (D.13) we find \\

\begin{align}
1-mn&=\frac{u_1(1+m^2)}{u_1+mu_2} \\
1+mn&=\frac{u_1(1-m^2)+2mu_2}{u_1+mu_2} \\
m-n&=\frac{2mu_1-u_2(1-m^2)}{u_1+mu_2} \\
u&=\frac{u_1+mu_2}{1+m^2}
\end{align}

and then \\

\begin{align}
u_3&=(u_2+u_1)\frac{2m}{1+m^2}-(u_2-u_1)\frac{1-m^2}{1+m^2} \\
u_4&=(u_2+u_1)\frac{1-m^2}{1+m^2}+(u_2-u_1)\frac{2m}{1+m^2}
\end{align} \\

This is a parameterization of a rational parallelogram by \\

\begin{align*}
u_1>0,u_2>0, \text{ and } 0<m<1
\end{align*} \\

From (D.14) we find \\

\begin{align}
1-mn&=\frac{u_1(1+n^2)}{u_1+nu_2} \\
1+mn&=\frac{u_1(1-n^2)+2nu_2}{u_1+nu_2} \\
m-n&=\frac{-2nu_1+u_2(1-n^2)}{u_1+nu_2} \\
u&=\frac{u_1+nu_2}{1+n^2}
\end{align}\\

and then \\

\begin{align}
u_3&=(u_2-u_1)\frac{2n}{1+n^2}+(u_2+u_1)\frac{1-n^2}{1+n^2} \\
u_4&=(u_2+u_1)\frac{2n}{1+n^2}-(u_2-u_1)\frac{1-n^2}{1+n^2}
\end{align}\\

This is a parameterization of a rational parallelogram by \\

\begin{align*}
u_1>0, u_2>0, \text{ and } 0<n<1
\end{align*}\\

Now the two parameters m,n in (D.3) give rise to the Heron angles ${\alpha}$, ${\beta}$ through \\

\begin{align}
\cos{\alpha}&=\frac{1-m^2}{1+m^2}, \hspace{0.25cm} \sin{\alpha}=\frac{2m}{1+m^2}, \hspace{0.25cm} 0<{\alpha}<\frac{\pi}{2} \\
\cos{\beta}&=\frac{1-n^2}{1+n^2}, \hspace{0.25cm} \sin{\beta}=\frac{2n}{1+n^2}, \hspace{0.25cm} 0<{\beta}<\frac{\pi}{2}
\end{align}\\

And conversely \\

\begin{align}
m&=\frac{\sin{\alpha}}{1+\cos{\alpha}}=\tan(\frac{{\alpha}}{2}) \\
n&=\frac{\sin{\beta}}{1+\cos{\beta}}=\tan(\frac{{\beta}}{2})
\end{align}\\

Using the ${\omega}$-functions we get in matrix notation the representation \\

\begin{align}
\left(
    \begin{array}{c}
      u_3 \\
      u_4 
    \end{array}
\right) = 
\left(
    \begin{array}{cc}
      {\omega}_+({\alpha}) & -{\omega}_-({\alpha}) \\
      {\omega}_-({\alpha}) & {\omega}_+({\alpha})
    \end{array}
\right)
\left(
    \begin{array}{c}
      u_1 \\
      u_2 
    \end{array}
\right) 
\end{align}

Inversely \\

\begin{align}
2\left(
    \begin{array}{c}
      u_1 \\
      u_2 
    \end{array}
\right) = 
\left(
    \begin{array}{cc}
      {\omega}_+({\alpha}) & {\omega}_-({\alpha}) \\
      -{\omega}_-({\alpha}) & {\omega}_+({\alpha})
    \end{array}
\right)
\left(
    \begin{array}{c}
      u_3 \\
      u_4 
    \end{array}
\right) 
\end{align}

and conversely \\

\begin{align}
{\omega}_+({\alpha})=\frac{1}{u_1^2+u_2^2}[u_1u_3+u_2u_4] \\
{\omega}_-({\alpha})=\frac{1}{u_1^2+u_2^2}[u_1u_4-u_2u_3]
\end{align} \\

Similarly we find the representation \\

\begin{align}
\left(
    \begin{array}{c}
      u_3 \\
      u_4 
    \end{array}
\right) = 
\left(
    \begin{array}{cc}
      {\omega}_-({\beta}) & {\omega}_+({\beta}) \\
      {\omega}_+({\beta}) & -{\omega}_-({\beta})
    \end{array}
\right)
\left(
    \begin{array}{c}
      u_1 \\
      u_2 
    \end{array}
\right) 
\end{align}

Inversely \\

\begin{align}
2\left(
    \begin{array}{c}
      u_1 \\
      u_2 
    \end{array}
\right) = 
\left(
    \begin{array}{cc}
      {\omega}_-({\beta}) & {\omega}_+({\beta}) \\
      {\omega}_+({\beta}) & -{\omega}_-({\beta})
    \end{array}
\right)
\left(
    \begin{array}{c}
      u_3 \\
      u_4 
    \end{array}
\right) 
\end{align} \\

and conversely \\
 
\begin{align}
{\omega}_+({\beta})=\frac{1}{u_1^2+u_2^2}[u_1u_4+u_2u_3] \\
{\omega}_-({\beta})=\frac{1}{u_1^2+u_2^2}[u_1u_3-u_2u_4] 
\end{align} \\

Finally we introduce the two Euler angles ${\sigma},{\delta}$ through \\

\begin{align}
{\alpha}+{\beta}=2{\sigma},& \hspace{0.25cm} {\alpha}-{\beta}=2{\delta} \\
{\alpha}={\sigma}+{\delta},& \hspace{0.25cm} {\beta}={\sigma}-{\delta}
\end{align}

From \\

\begin{align}
u_1{\omega}_+({\alpha})-u_2{\omega}_-({\alpha})=u_1{\omega}_-({\beta})+u_2{\omega}_+({\beta})
\end{align}

and 

\begin{align}
2u_3&=u_1{\omega}_+({\alpha})-u_2{\omega}_-({\alpha})+u_1{\omega}_-({\beta})+u_2{\omega}_+({\beta}) \\
2u_4&=u_1{\omega}_-({\alpha})+u_2{\omega}_+({\alpha})+u_1{\omega}_+({\beta})-u_2{\omega}_-({\beta})
\end{align}\\

we find, using (C.19), (C.22), (C.24), (C.25) the relations

\begin{align}
&u_1\sin{\sigma}=u_2\cos{\sigma} \\
&u_3=[u_1\cos{\sigma}+u_2\sin{\sigma}]{\omega}_+({\delta}) \\
&u_4=[u_1\cos{\sigma}+u_2\sin{\sigma}]{\omega}_-({\delta})
\end{align}\\

resulting in the representation \\

\begin{align}
u_2&=u_1\tan{\sigma} \\
u_3&=u_1\frac{{\omega}_+({\delta})}{\cos{\sigma}} \\
u_4&=u_1\frac{{\omega}_-({\delta})}{\cos{\sigma}}
\end{align}

Conversely \\

\begin{align}
\tan{\sigma}&=\frac{u_2}{u_1} \\
\tan{\delta}&=\frac{u_3-u_4}{u_3+u_4}
\end{align}\\

From (D.27), (D.28), (D.39), (D.40) we find the inequalities \\

\begin{align}
\tan{\sigma}>0 \\
-1<\tan{\delta}<1
\end{align}\\

(D.47-49) is a parameterization of a rational parallelogram, given one side $u_1>0$, by two Euler angles. \\

Observe, that given $u_1$ and $u_2$ there are two representations of $u_3$ and $u_4$, (D.31) (D.35). One is through the Heron angle ${\alpha}$, with generator $m = m({\alpha})$ and the other through the Heron angle ${\beta}$, with generator $n = m({\beta})$. According to (D.8) and (D.9) they are however related through the interchange of $u_3$ with $u_4$.

\clearpage
\section*{Appendix E}
\subsection*{\centerline{Auxiliary Functions}}

\setcounter{equation}{0}
\renewcommand{\theequation}{E.\arabic{equation}}

For an angle ${\alpha}$ and a number Q we introduce the functions

\begin{align}
H({\alpha}) &= {\omega}_-({\alpha})-Q{\omega}_+({\alpha}) \\
K({\alpha}) &= {\omega}_-({\alpha})+Q{\omega}_+({\alpha}) \\
M({\alpha}) &= {\omega}_+({\alpha})-Q{\omega}_-({\alpha}) \\
N({\alpha}) &= {\omega}_+({\alpha})+Q{\omega}_-({\alpha})
\end{align}

or in matrix form \\

\begin{align}
\left(
    \begin{array}{c}
      H({\alpha}) \\
      N({\alpha}) 
    \end{array}
\right) = 
\left(
    \begin{array}{cc}
      {\omega}_-({\alpha}) & -{\omega}_+({\alpha}) \\
      {\omega}_+({\alpha}) & {\omega}_-({\alpha})
    \end{array}
\right)
\left(
    \begin{array}{c}
      1 \\
      Q 
    \end{array}
\right) =
R({\alpha})
\left(
    \begin{array}{c}
      1-Q \\
      1+Q 
    \end{array}
\right)
\end{align}

\begin{align}
\left(
    \begin{array}{c}
      K({\alpha}) \\
      M({\alpha}) 
    \end{array}
\right) = 
\left(
    \begin{array}{cc}
      {\omega}_-({\alpha}) & {\omega}_+({\alpha}) \\
      {\omega}_+({\alpha}) & -{\omega}_-({\alpha})
    \end{array}
\right)
\left(
    \begin{array}{c}
      1 \\
      Q 
    \end{array}
\right) =
R({\alpha})
\left(
    \begin{array}{c}
      1+Q \\
      1-Q 
    \end{array}
\right)
\end{align} \\

We now introduce the T-matrix \\

\begin{align}
T=
\left(
    \begin{array}{cc}
      0 & 1 \\
      1 & 0 
    \end{array}
\right)
\end{align} \\

and find the following properties \\

\begin{align}
T^2=
\left(
    \begin{array}{cc}
      1 & 0 \\
      0 & 1 
    \end{array}
\right)
\end{align}

\begin{align}
TR({\alpha})T=R(-{\alpha}) 
\end{align}

\begin{align}
T\binom{x}{y}=\binom{y}{x}
\end{align}

\begin{align}
R({\alpha})TR(-{\alpha})=
\left(
    \begin{array}{cc}
      -\sin(2{\alpha}) & \cos(2{\alpha}) \\
      \cos(2{\alpha}) & \sin(2{\alpha}) 
    \end{array}
\right)
\end{align} \\

Then we have the following statements \\

\subsubsection*{\underline{Lemma 1}}

\begin{align}
2&=H({\alpha}){\omega}_-({\alpha})+N({\alpha}){\omega}_+({\alpha}) \\
2&=K({\alpha}){\omega}_-({\alpha})+M({\alpha}){\omega}_+({\alpha}) \\
2Q&=N({\alpha}){\omega}_-({\alpha})-H({\alpha}){\omega}_+({\alpha}) \\
2Q&=K({\alpha}){\omega}_+({\alpha})-M({\alpha}){\omega}_-({\alpha})
\end{align}

\subsubsection*{\underline{Lemma 2}}

\begin{align}
\left(
    \begin{array}{c}
      H({\alpha}+{\beta}) \\
      N({\alpha}+{\beta})
    \end{array}
\right) = R({\alpha})
\left(
    \begin{array}{c}
      H({\beta}) \\
      N({\beta})
    \end{array}
\right) = R({\beta})
\left(
    \begin{array}{c}
      H({\alpha}) \\
      N({\alpha})
    \end{array}
\right)
\end{align}

\begin{align}
\left(
    \begin{array}{c}
      K({\alpha}+{\beta}) \\
      M({\alpha}+{\beta})
    \end{array}
\right) = R({\alpha})
\left(
    \begin{array}{c}
      K({\beta}) \\
      M({\beta})
    \end{array}
\right) = R({\beta})
\left(
    \begin{array}{c}
      K({\alpha}) \\
      M({\alpha})
    \end{array}
\right)
\end{align}

\subsubsection*{\underline{Lemma 3}}

With ${\bar{\alpha}}=\frac{\pi}{2}-{\alpha}$ we have \\

\begin{align}
H({\bar{\alpha}})=-K({\alpha}), \hspace{0.25cm} N({\bar{\alpha}}) = M({\alpha})
\end{align}

\subsubsection*{\underline{Lemma 4}}

\begin{align}
\left(
    \begin{array}{c}
      H({\alpha}) \\
      N({\alpha})
    \end{array}
\right) = R(2{\alpha})
\left(
    \begin{array}{c}
      M({\alpha}) \\
      K({\alpha})
    \end{array}
\right) \\
\left(
    \begin{array}{c}
      K({\alpha}) \\
      M({\alpha})
    \end{array}
\right) = R(2{\alpha})
\left(
    \begin{array}{c}
      N({\alpha}) \\
      H({\alpha})
    \end{array}
\right) 
\end{align}

\subsubsection*{\underline{Lemma 5}}

\begin{align}
1) \hspace{0.5cm} &H^2({\alpha})+N^2({\alpha}) = K^2({\alpha})+M^2({\alpha}) = 2(1+Q^2) \\
2) \hspace{0.5cm} &K({\alpha})N({\alpha})-H({\alpha})M({\alpha}) = 4Q \\
3) \hspace{0.5cm} &K({\alpha})N({\alpha})+H({\alpha})M({\alpha}) = 2(1+Q^2)\cos(2{\alpha}) \\
4) \hspace{0.5cm} &K({\alpha}+{\delta})H({\alpha}+{\beta})-K({\alpha}+{\beta})H({\alpha}+{\delta}) = 4Q\sin({\delta}-{\beta}) \\
5) \hspace{0.5cm} &N({\alpha}+{\beta})M({\alpha}+{\delta})-N({\alpha}+{\delta})M({\alpha}+{\beta}) = 4Q\sin({\delta}-{\beta})
\end{align}

\subsubsection*{\underline{Lemma 6}}

\begin{align}
1) \hspace{0.5cm} &K({\alpha})+H({\alpha}) = 2{\omega}_-({\alpha}) \\
2) \hspace{0.5cm} &K({\alpha})-H({\alpha}) = 2Q{\omega}_+({\alpha}) \\
3) \hspace{0.5cm} &N({\alpha})+M({\alpha}) = 2{\omega}_+({\alpha}) \\
4) \hspace{0.5cm} &N({\alpha})-M({\alpha}) = 2Q{\omega}_-({\alpha}) \\
5) \hspace{0.5cm} &M({\alpha})+H({\alpha}) = 2\cos{\alpha}[1-Q] \\
6) \hspace{0.5cm} &M({\alpha})-H({\alpha}) = 2\sin{\alpha}[1+Q] \\
7) \hspace{0.5cm} &N({\alpha})+K({\alpha}) = 2\cos{\alpha}[1+Q] \\
8) \hspace{0.5cm} &N({\alpha})-K({\alpha}) = 2\sin{\alpha}[1-Q]
\end{align} \\

Now with

\begin{align*}
{\alpha}+{\beta}=2{\sigma},& \hspace{0.25cm} {\alpha}-{\beta}=2{\delta} \\
{\alpha}={\sigma}+{\delta},& \hspace{0.25cm} {\beta}={\sigma}-{\delta}
\end{align*} \\

we introduce the angle ${\psi}$ through

\begin{align}
{\omega}_+({\sigma})=\sqrt{2}\cos{\psi}, \hspace{0.25cm} {\omega}_-({\sigma})=\sqrt{2}\sin{\psi}
\end{align} \\

Then ${\psi}$ is an Euler angle and for 

\begin{align*}
0<{\alpha}<\frac{{\pi}}{2}, \hspace{0.25cm} 0<{\beta}<\frac{{\pi}}{2}, \hspace{0.25cm} \text{ we find } -\frac{{\pi}}{4}<{\psi}<\frac{{\pi}}{4}
\end{align*} \\

From ${\sigma}={\alpha}-{\delta}$ and (C.13) we find \\

\begin{align}
\left(
    \begin{array}{c}
      {\omega}_+({\sigma}) \\
      {\omega}_-({\sigma})
    \end{array}
\right) = R({\delta})
\left(
    \begin{array}{c}
      {\omega}_+({\alpha}) \\
      {\omega}_-({\alpha})
    \end{array}
\right) = R({\psi})
\left(
    \begin{array}{c}
      \sqrt{2} \\
      0
    \end{array}
\right)
\end{align}

\begin{align}
\left(
    \begin{array}{c}
      {\omega}_+({\alpha}+{\psi}) \\
      {\omega}_-({\alpha}+{\psi})
    \end{array}
\right) = R(-{\psi})
\left(
    \begin{array}{c}
      {\omega}_+({\alpha}) \\
      {\omega}_-({\alpha})
    \end{array}
\right) = R(-{\delta})
\left(
    \begin{array}{c}
      \sqrt{2} \\
      0
    \end{array}
\right)
\end{align} \\

Explicitly \\

\begin{align}
\sqrt{2}\cos{\delta}={\omega}_+({\alpha}+{\psi}), \hspace{0.25cm} \sqrt{2}\sin{\delta} = -{\omega}_-({\alpha}+{\psi})
\end{align} \\

\subsubsection*{\underline{Lemma 7}}

\begin{align}
1) \hspace{0.5cm} M({\alpha})-N({\beta}) &= -2\sin{\psi}K({\alpha}+{\psi}) \\
2) \hspace{0.5cm} M({\alpha})+N({\beta}) &= 2\cos{\psi}M({\alpha}+{\psi}) \\
3) \hspace{0.5cm} N({\alpha})-M({\beta}) &= -2\sin{\psi}H({\alpha}+{\psi}) \\
4) \hspace{0.5cm} N({\alpha})+M({\beta}) &= 2\cos{\psi}N({\alpha}+{\psi}) \\
5) \hspace{0.5cm} K({\alpha})+H({\beta}) &= 2\sin{\psi}M({\alpha}+{\psi}) \\
6) \hspace{0.5cm} K({\alpha})-H({\beta}) &= 2\cos{\psi}K({\alpha}+{\psi}) \\
7) \hspace{0.5cm} H({\alpha})+K({\beta}) &= 2\sin{\psi}N({\alpha}+{\psi}) \\
8) \hspace{0.5cm} H({\alpha})-K({\beta}) &= 2\cos{\psi}H({\alpha}+{\psi})
\end{align}

The easy proofs are left to the reader. \\

\clearpage
\section*{Appendix F}
\subsection*{\centerline{A Convenient Bijective Parameterization} \\}

\setcounter{equation}{0}
\renewcommand{\theequation}{F.\arabic{equation}}

For a given Heron angle ${\alpha}$, $0<{\alpha}<\frac{\pi}{2}$, we look at the quadratic equation

\begin{equation}
\begin{aligned}
(M^2-1)\sin(2{\alpha}) + 2M\cos(2{\alpha}) = 4D, M>0
\end{aligned}
\end{equation}

or

\begin{equation}
\begin{aligned}
M^2\sin(2{\alpha}) + 2M\cos(2{\alpha}) - [\sin(2{\alpha})+4D] = 0
\end{aligned}
\end{equation}

Then the solutions are given by

\begin{equation}
\begin{aligned}
M_+ = \frac{1}{\sin(2{\alpha})} [-\cos(2{\alpha}) + \Delta] 
\end{aligned}
\end{equation}

\begin{equation}
\begin{aligned}
M_- = \frac{1}{\sin(2{\alpha})} [-\cos(2{\alpha}) - \Delta]
\end{aligned}
\end{equation}

where 

\begin{equation}
\begin{aligned}
{\Delta}^2 =& {\cos^2(2{\alpha})} + \sin(2{\alpha}) [\sin(2{\alpha}) + 4D] \\
{\Delta}^2 =& 1 + 4\sin(2{\alpha})D, {\Delta} \geq 0
\end{aligned}
\end{equation}

or

\begin{equation}
\begin{aligned}
{\Delta}^2 + [\sin(2{\alpha})-D]^2 = 1+[\sin(2{\alpha})+D]^2
\end{aligned}
\end{equation}

According to [8] this relation has the following bijective parameter representation with the parameters $a_1, a_2, {\lambda}$ 

\begin{equation}
\begin{aligned}
{\Delta} &= a_1+{\lambda}a_2 &1&=a_1-{\lambda}a_2
\end{aligned}
\end{equation}

\begin{equation}
\begin{aligned}
\sin(2{\alpha}) - D &= a_2-{\lambda}a_1 &\sin(2{\alpha})+D &= a_2+{\lambda}a_1
\end{aligned}
\end{equation}

Conversely 

\begin{equation}
\begin{aligned}
2a_1 &= 1+\Delta
\end{aligned}
\end{equation}

\begin{equation}
\begin{aligned}
a_2&=\sin(2{\alpha})
\end{aligned}
\end{equation}

\begin{equation}
\begin{aligned}
{\lambda}a_1 &= D
\end{aligned}
\end{equation}

From (F.5) and (F.9) we find that

\begin{equation}
\begin{aligned}
a_1 \neq 0,\hspace{0.25cm}  1+{\lambda}a_2=a_1 \neq 0
\end{aligned}
\end{equation}

Then

\begin{equation}
\begin{aligned}
&a_1=1+{\lambda}a_2 = 1+{\lambda}\sin(2{\alpha})
\end{aligned}
\end{equation}

\begin{equation}
\begin{aligned}
&\Delta=1+2{\lambda}a_2 = 1+2{\lambda}\sin(2{\alpha})
\end{aligned}
\end{equation}

and from (F.3), (F.4) the equation (F.1) is bijectively parameterized by the parameter ${\lambda}$ as

\begin{equation}
\begin{aligned}
&M_+=2{\lambda}+\tan{\alpha} \\
\end{aligned}
\end{equation}

\begin{equation}
\begin{aligned}
&M_-=-2{\lambda}-\cot{\alpha} \\
\end{aligned}
\end{equation}

\begin{equation}
\begin{aligned}
&D={\lambda}[1+{\lambda}\sin(2{\alpha})]
\end{aligned}
\end{equation}

Conversely

\begin{equation}
\begin{aligned}
2{\lambda}=M_+-\tan{\alpha}
\end{aligned}
\end{equation}

\begin{remark}
\hfill \break

1) From (F.13) we find

\begin{equation}
\begin{aligned}
1+{\lambda}\sin(2{\alpha}) \neq 0
\end{aligned}
\end{equation}

2) For the limit

\begin{equation}
\begin{aligned}
{\lambda} \rightarrow 0
\end{aligned}
\end{equation}

we find from (F.15-17) that

\begin{equation}
\begin{aligned}
&M_+ \rightarrow \tan{\alpha} \\
\end{aligned}
\end{equation}

\begin{equation}
\begin{aligned}
&M_- \rightarrow -\cot{\alpha} \\
\end{aligned}
\end{equation}

\begin{equation}
\begin{aligned}
&D \rightarrow 0 
\end{aligned}
\end{equation}

Since the solution of (F.1) has to be positive, the only acceptable solution for the limit (F.20) is 

\begin{equation}
\begin{aligned}
M = M_+
\end{aligned}
\end{equation}

\end{remark}

\bigskip

\noindent\textit{Department of Physics, University of Colorado Boulder, Boulder, CO 80309\\
Walter.Wyss@Colorado.EDU}

\section*{Acknowledgement}
I would like to thank my family for encouragement, especially my wife Yvonne for her patience
and our son Daniel for discussions and typing the manuscript. 

\end{document}